\documentclass[a4paper,12pt]{article}
\usepackage{times}
\usepackage[english]{babel}
\usepackage{amssymb,amsmath}
\numberwithin{equation}{section}
\usepackage{amsthm}
\usepackage{array}
\usepackage{wasysym}
\usepackage[noadjust]{cite}
\usepackage{hyperref}
\usepackage[height=22.7cm , width = 16cm , top = 4cm , left = 3cm, a4paper]{geometry}
\usepackage{amssymb}
\usepackage{amsmath}
\usepackage{cite}
\usepackage{epsfig}
\usepackage{verbatim}
\usepackage{multicol}
\usepackage{graphicx, color, psfrag}
\usepackage{graphics, psfrag}
\usepackage{subfigure}
\usepackage[numbers,sort]{natbib}
 \newtheorem{theorem}{Theorem}[section]
\newtheorem{remark}[theorem]{Remark}

\newtheorem{corollary}{Corollary}[section]

\theoremstyle{definition}

\newcommand{\e}{\end{document}}

\begin{document}

\thispagestyle{empty}

\author{
{{\bf B. S. El-Desouky\footnote{Corresponding author: b\_desouky@yahoo.com} \; and Abdelfattah Mustafa}
\newline{\it{{}  }}
 } { }\vspace{.2cm}\\
 \small \it Department of Mathematics, Faculty of Science, Mansoura University,Mansoura 35516, Egypt
}

\title{New Results on Higher-Order Daehee and Bernoulli Numbers and Polynomials}

\date{}

\maketitle
\small \pagestyle{myheadings}
        \markboth{{\scriptsize New Results on Higher-Order Daehee and Bernoulli Numbers and Polynomials }}
        {{\scriptsize {El-Desouky and Mustafa}}}

\hrule \vskip 8pt

\begin{abstract}
 We derive new matrix representation for higher order Daehee numbers and polynomials, the higher order $\lambda$-Daehee numbers and polynomials and the twisted $\lambda$-Daehee numbers and polynomials of order $k$. This helps us to obtain simple and short proofs of many previous results on higher order Daehee numbers and polynomials. Moreover, we obtained recurrence relation, explicit formulas and some new results for these numbers and polynomials. Furthermore, we investigated the relation between these numbers and polynomials and Stirling numbers, N\"{o}rlund and Bernoulli numbers of higher order. The results of this article gives a generalization of the results derived very recently by El-Desouky and Mustafa \cite{El-DesoukyMustafa2014}.
\end{abstract}




\section{Introduction}
For $\alpha \in \mathbb{N}$,  the Bernoulli polynomials of order $\alpha$ are defined by, see \cite{Carlitz1961}-\cite{Ozdenetal2009}
\begin{equation}\label{1}
\left(\frac{t}{e^t-1}\right)^{\alpha} e^{xt}= \sum_{n=0}^\infty B_n^{(\alpha) } (x) \frac{  t^n}{n!}.
\end{equation}

\noindent
When $x=0,B_n^{(\alpha) }=B_n^{(\alpha)} (0)$ are the Bernoulli numbers of order $\alpha$, defined by
\begin{equation}\label{2}
\left( \frac{t}{(e^t-1} \right)^{\alpha}=\sum_{n=0}^\infty B_n^{(\alpha) }  \frac{ t^n}{n!}.
\end{equation}

\noindent
The Daehee polynomials are defined by, see \cite{KimKim2013, KimSimsek2008} and  \cite{Ozdenetal2009}.
\begin{equation} \label{3}
\left( \frac{\log⁡(1+t)}{t} \right) (1+t)^x = \sum_{n=0}^\infty D_n (x) \frac{ t^n}{n!}.
\end{equation}

\noindent
In the special case, $x=0,D_n=D_n (0)$ are called the Daehee numbers, defined by
\begin{equation} \label{4}
\left(\frac{\log⁡(1+t)⁡}{t} \right)=\sum_{n=0}^\infty D_n  \frac{ t^n}{n!}.
\end{equation}

\noindent
The Stirling numbers of the first and second kind are defined, respectively, by
\begin{equation} \label{5}
(x)_n = \prod_{i=0}^n (x-i) = \sum_{l=0}^n s_1 (n,l) x^l,
\end{equation}
where $s_1(n,0)=\delta_{n,0}, \; s_1(n,k)=0, \; \mbox{for} \; k>n,$, and
\begin{equation} \label{5a}
x^n =  \sum_{k=0}^n s_2 (n,k) (x)_k,
\end{equation}
where $s_2(n,0)=\delta_{n,0}, \; s_2(n,k)=0, \; \mbox{for} \; k>n$, and $\delta_{n,k}$ is the kronecker delta.

\noindent
The Stirling numbers of the second kind  have the generating function, see  \cite{Carlitz1961,Comtet1974,El-Desouky1994,El-Desoukyetal2010} and \cite{Gould1972}.
\begin{equation} \label{6}
\left(e^t-1 \right)^m = m! \sum_{l=m}^\infty s_2 (l,m) \frac{  t^l}{l!}.
\end{equation}


\section{Higher order Daehee Numbers and Polynomials}

In this section, we derive an explicit formulas and recurrence relations for the higher order Daehee numbers and polynomials of the first and second kinds. Also the relation between these numbers and N\"{o}rlund numbers are given. Furthermore, we introduce the matrix representation of some results for higher order Daehee numbers and polynomial obtained by Kim et al. \cite{Kimetal2014} in terms of Stirling numbers, N\"{o}rlund numbers and Bernoulli numbers of higher order and investigate a simple and short proofs of these results.

\noindent
Kim et al. \cite{Kimetal2014} defined the Daehee numbers of the first kind of order $k$, by the following generating function
\begin{equation} \label{7}
\sum_{n=0}^\infty D_n^{(k)}   \frac{t^n}{n!}=\left(\frac{\log⁡(1+t)}{t} \right)^k.
\end{equation}

\noindent
Next, an explicit formula for $D_n^{(k) }$ is given by the following theorem.
\begin{theorem}
For $n\in \mathbb{Z},\; k \in \mathbb{N}$, we have
\begin{equation}\label{8}
D_n^{(k) } = n! \sum_{l_1+l_2+ \cdots +l_k=n+k } \frac{(-1)^n}{l_1 l_2 \cdots l_k }.
\end{equation}
\end{theorem}
\begin{proof}
From Eq. (\ref{7}), we have
\begin{equation*}
\sum_{n=0}^\infty D_n^{(k)}   \frac{t^{n+k}}{n!}=\left(\log⁡(1+t) \right)^k=\left( \sum_{l=1}^\infty \frac{(-1)^{l-1}}{l}  t^l \right)^k.
\end{equation*}

\noindent
Using Cauchy rule of product of series, we obtain
\begin{equation*}
\sum_{n=0}^\infty D_n^{(k)}   \frac{t^{n+k}}{n!} = \sum_{r=k}^\infty \sum_{l_1+l_2+ \cdots +l_k=r}^\infty \frac{(-1)^{r-k}}{l_1 l_2 \cdots l_k }  t^r,
\end{equation*}
let $r-k=n$, in the right hand side, we have

\begin{equation*}
\sum_{n=0}^\infty D_n^{(k)}   \frac{t^{n+k}}{n!}= \sum_{n=0}^\infty \sum_{l_1+l_2+ \cdots +l_k=n+k}^\infty \frac{(-1)^n}{l_1 l_2 \cdots l_k }  t^{n+k}.
\end{equation*}
Equating the coefficients of $t^{n+k}$ on both sides yields (\ref{8}). This completes the proof.
\end{proof}

\begin{remark}
It is worth noting that setting $k=1$ in (\ref{8}), we get \cite[ Eq. (2.2)]{El-DesoukyMustafa2014} as a special case.
\end{remark}

\noindent
Kim et al. \cite[2014, Theorem 1]{Kimetal2014} proved that, see \cite{Wang2010}, for $n \in \mathbb{Z}, \; k\in \mathbb{N}$, we have
\begin{equation} \label{9}
D_n^{(k) }=\frac{s_1 (n+k,k)}{\left(^{n+k}_{\; \; \; k} \right)}.
\end{equation}

\noindent
We can represent the Daehee numbers of the first kind of order $k$,  by $(n+1)\times(k+1)$ matrix , $0\leq k \leq n$, as follows
\begin{equation*}
{\bf D}^{(k) }=
\left(
\begin{array}{ccccc}
D_0^{(0)} & D_0^{(1)} & D_0^{(2)} & \cdots & D_0^{(k)} \\
D_1^{(0)} & D_1^{(1)} & D_1^{(2)} & \cdots & D_1^{(k)} \\
\vdots    & \vdots    & \vdots    & \ddots & \vdots    \\
D_n^{(0) }& D_n^{(1) }& D_n^{(2) } & \cdots & D_n^{(k) }
\end{array}
\right).
\end{equation*}
For example if $0\leq n \leq 3,\; 0 \leq k \leq n$, we have
\begin{equation*}
{\bf D}^{(k)}=
\left(
\begin{array}{cccc}
1 & 1 & 1 & 1 \\
0 &-1/2 & -1 &-3/2 \\
0 & 2/3 & 11/6 & 7/2 \\
0 & -3/2 & -5 & -45/4 \\
\end{array}
\right).
\end{equation*}

\noindent
Kim et al. \cite[Theorem 4]{Kimetal2014},  proved the following result.
For $n \in \mathbb{Z}, \; k\in \mathbb{N}$, we have
\begin{equation}\label{10}
B_n^{(k)}=\sum_{m=0}^n D_m^{(k)}  s_2 (n,m).
\end{equation}

\noindent
\begin{remark}
We can write this relation in the matrix form as follows.
\begin{equation}\label{11}
{\bf B}^{(k)}= {\bf S}_2 \, {\bf D}^{(k)},
\end{equation}
\end{remark}

\noindent
where ${\bf D}^{(k) }$  is $(n+1)\times (k+1), \, 0\leq k\leq n,$ matrix for the Daehee numbers of the first kind of order $k$ and  ${\bf S}_2$ is  $(n+1)\times(n+1)$ lower triangular matrix for the Strirling numbers of the second kind and ${\bf B}^{(k) }$ is $(n+1)\times(k+1), \, 0\leq k \leq n,$ matrix for the Bernoulli numbers of order $k$.\\
For example, if setting  $0\leq n \leq 3, \; 0 \leq k \leq n$ , in (\ref{11}), we have
\begin{equation*}
\left(
\begin{array}{cccc}
1 & 0 & 0 & 0 \\
0 & 1 & 0 & 0 \\
0 & 1 & 1 & 0 \\
0 & 1 & 3 & 1
\end{array}
\right)\left(
\begin{array}{ccccc}
1 & 1 & 1 & 1 \\
0 & -1/2 & -1 & -3/2 \\
0 & 2/3 & 11/6 & 7/2 \\
0 & -3/2 & -5 & -45/4
\end{array}
\right)=\left(
\begin{array}{ccccc}
1 & 1 & 1 & 1 \\
0 & -1/2 & -1 &-3/2 \\
0 & 1/6 & 5/6 & 2 \\
0 & 0 & -1/2 & -9/4
\end{array}
\right).
\end{equation*}

\noindent
Kim et al. \cite[Theorem 3]{Kimetal2014} introduced the following result.
For $n\in \mathbb{Z},\;  k\in \mathbb{N}$, we have
\begin{equation} \label{12}
D_n^{(k)}= \sum_{m=0}^n  s_1 (n,m) B_m^{(k)}.
\end{equation}

\noindent
We can write this relation in the matrix form as follows
\begin{equation} \label{13}
{\bf D}^{(k)}={\bf S}_1 \, {\bf B}^{(k)},
\end{equation}

\noindent
where  ${\bf S}_1$ is $(n+1)\times(n+1)$ lower triangular matrix for the Strirling numbers of the first kind.

\noindent
For example, if setting  $0\leq n \leq 3, \; 0 \leq k \leq n$ , in (\ref{13}), we have
\begin{equation*}
\left(
\begin{array}{cccc}
1 & 0 & 0 & 0 \\
0 & 1 & 0 & 0 \\
0 & -1 & 1 & 0 \\
0 & 2 & -3 & 1
\end{array}
\right)\left(
\begin{array}{cccc}
1 & 1 & 1 & 1 \\
0 & -1/2 & -1 & -3/2 \\
0 & 1/6 & 5/6 & 2 \\
0 & 0 & -1/2 & -9/4
\end{array}
\right)=
\left(
\begin{array}{cccc}
1 & 1 & 1 & 1 \\
0 & -1/2 & -1 & -3/2 \\
0 & 2/3 & 11/6 & 7/2 \\
0 & -3/2 & -5 & -45/4
\end{array}
\right).
\end{equation*}

\begin{remark}
Using the matrix form (\ref{13}), we easily derive a short proof of Theorem 4 in Kim et al. \cite{Kimetal2014}. Multiplying both sides by the Striling number of second kind as follows.
\begin{equation*}
{\bf S}_2 \, {\bf D}^{(k) }= {\bf S}_2 \, {\bf S}_1 \,  {\bf B}^{(k)}={\bf I\, B}^{(k)}={\bf B}^{(k) },
\end{equation*}

\noindent
where {\bf I} is the identity matrix of order $(n+1)$.
\end{remark}

\noindent
Kim et al. \cite{Kimetal2014} defined the Daehee polynomials of order $k$ by the generating function as follows.
\begin{equation} \label{14}
\sum_{n=0}^\infty D_n^{(k) }  (x)  \frac{t^n}{n!}=\left(\frac{\log⁡(1+t)}{t}\right)^k (1+t)^x.
\end{equation}

\noindent
Liu and Srivastava \cite{LiuSrivastava2006} define the N\"{o}rlund numbers of the second kind $b_n^{(x) }$ as follows.
\begin{equation} \label{15}
\left(\frac{t}{\log⁡(1+t)} \right)^x = \sum_{n=0}^\infty b_n^{(x) }  t^n.
\end{equation}

\noindent
Next, we find the relation between the Daehee polynomials of order $k$ and the N\"{o}rlund numbers of the second kind $b_n^{(x)}$ by the following theorem.
\begin{theorem}
For $m\in \mathbb{Z},\;  k\in \mathbb{N}$, we have
\begin{equation} \label{16}
D_m^{(k) } (z)=m! \sum_{n=0}^m \left(^{\; \; \;z}_{m-n} \right) b_n^{(-k)}.
\end{equation}
\end{theorem}
\begin{proof}
From Eq. (\ref{15}), by multiplying both sides by $(1+t)^z$, we have
\begin{eqnarray}\label{17}
\left(\frac{t}{\log⁡(1+t)}\right)^x (1+t)^z & = & \sum_{n=0}^\infty b_n^{(x)}  t^n (1+t)^z =  \sum_{n=0}^\infty b_n^{(x) }  t^n \sum_{i=0}^\infty \left(^z_i \right) t^i
\nonumber\\
&= & \sum_{n=0}^\infty \sum_{m=n}^\infty b_n^{(x)} \left(^{\;\; \;z}_{m-n} \right) t^m
=  \sum_{m=0}^\infty \sum_{n=0}^m \left(^{\; \;\; z}_{m-n}\right) b_n^{(x) } t^m.
\end{eqnarray}

\noindent
Replacing $x$ by $-k$ in (\ref{17}), we have
\begin{eqnarray} \label{18}
\left(\frac{\log⁡(1+t)}{t}\right)^k (1+t)^z & = & \sum_{m=0}^\infty \sum_{n=0}^m \left(^{\; \;\;z}_{m-n} \right) b_n^{(-k)} t^m.
\nonumber\\
& = & \sum_{m=0}^\infty m! \sum_{n=0}^m \left(^{\;\; \; z}_{m-n} \right) b_n^{(-k) } \frac{ t^m}{m!}.
\end{eqnarray}

\noindent
From  (\ref{14}) and (\ref{18}), we have (\ref{16}). This completes the proof.
\end{proof}

\noindent
\begin{corollary}
Setting  $k=1$ in (\ref{16}) we have
\begin{equation} \label{19}
D_m (z) = m! \sum_{n=0}^m \left(^{\; \; \; z}_{m-n} \right) b_n^{(-1)}.
\end{equation}
\end{corollary}

\noindent
Setting $z=0$, in (\ref{16}), we have the following relation between Daehee numbers of higher order and N\"{o}rlund numbers of the second kind.
\begin{corollary}.
For $k\in \mathbb{N}$, by setting $z=0$ in (\ref{16})  we obtain
\begin{equation} \label{20}
D_m^{(k)}=m!\, b_m^{(-k)}.
\end{equation}
\end{corollary}

\noindent
The relation between the Bernoulli numbers and Bernoulli polynomials of order k are given by Kimura \cite{Kimura2003}, as follows.
\begin{equation} \label{21}
B_n^{(k)} (x) = \sum_{j=0}^n \left(^n_j\right)  B_j^{(k)} x^{n-j}.
\end{equation}

\noindent
Therefore, we can represent (\ref{21}) in the matrix form
\begin{equation} \label{22}
{\bf B}^{(k)} (x)= {\bf P}(x) \,  {\bf B}^{(k) },
\end{equation}

\noindent
where ${\bf B}^{(k)} (x)$  is $(n+1)\times(k+1)$ matrix, $0\leq k \leq n$ for Bernoulli polynomials of order $k$ as follows
\begin{equation*}
{\bf B}^{(k) } (x)=
\left(
\begin{array}{ccccc}
B_0^{(0)}(x) & B_0^{(1)}(x) & B_0^{(2)}(x) & \cdots & B_0^{(k)}(x)\\
B_1^{(0)}(x) & B_1^{(1)}(x) & B_1^{(2)}(x) & \cdots & B_1^{(k)}(x)\\
\vdots       & \vdots       &  \vdots      &\ddots  &\vdots\\
B_n^{(0)}(x) & B_n^{(1)}(x) & B_n^{(2)}(x) & \cdots & B_n^{(k)}(x))
\end{array}
\right),
\end{equation*}
where the column $k$ represents the Bernoulli polynomials of order $k$, ${\bf B}^{(k)}$ is $(n+1)\times(k+1)$ matrix, $0\leq k \leq n$ for Bernoulli numbers of order $k$ and the matrix ${\bf P}(x)$, the Pascal matrix, is $(n+1)\times(n+1)$ lower triangular matrix defined by
\begin{equation*}
({\bf P}(x))_{ij}=
\left\{
\begin{array}{cl}
\left(^i_j \right) x^{i-j}, &i\geq j,\\
0, & \mbox{otherwise}
\end{array}
\right.
, \;   i,j=0,1,\cdots,n.
\end{equation*}

\noindent
For example if setting $0 \leq n \leq 3, \; 0 \leq k \leq n$ in (\ref{22}), we have
\begin{eqnarray*}
& &
\left(
\begin{array}{cccc}
1 & 0 & 0 & 0 \\
x & 1 & 0 & 0 \\
x^2 & 2x & 1 & 0 \\
x^3 & 3x^2 & 3x & 1
\end{array}
\right)\left(
\begin{array}{cccc}
1 & 1 & 1 & 1 \\
0 & -1/2 & -1 &-3/2 \\
0 & 1/6 & 5/6 & 2 \\
0 & 0 & -1/2 & -9/4
\end{array}
\right)=
\\
&&
\; \; \; \; \; \; \; \; \; \; \; \; \; \; \; \;\; \; \; \; \; \; \; \; \; \; \; \; \; \; \; \;
\left(
\begin{array}{cccc}
1 & 1 & 1 & 1 \\ x & x- \frac{1}{2} & x-1 & x-\frac{3}{2} \\ x^2 & x^2-x+\frac{1}{6} & x^2-2x+\frac{5}{6} & x^2-3x+2 \\ x^3 & x^3-\frac{3}{2}x^2 + \frac{1}{2}x & x^3-3x^2+\frac{5}{2}x-\frac{1}{2} & x^3-\frac{9}{2}x^2+6x-\frac{9}{4}
\end{array}
\right).
\end{eqnarray*}

\noindent
Kim et al. \cite[Theorem 5]{Kimetal2014} introduced the following result. For $n\in \mathbb{Z}, \; k\in \mathbb{N}$,
\begin{equation} \label{23}
D_n^{(k)} (x) = \sum_{m=0}^n  s_1 (n,m) B_m^{(k)} (x).
\end{equation}

\noindent
We can write this relation in the matrix form as follows
\begin{equation} \label{24}
{\bf D}^{(k)} (x) = {\bf S}_1 \,  {\bf B}^{(k)} (x),
\end{equation}

\noindent
where ${\bf D}^{(k) } (x)$ is $(n+1)\times (k+1)$ matrix for the Daehee polynomials of the first kind with order $k$ and  ${\bf B}^{(k) } (x)$ is $(n+1)\times(k+1)$ matrix for the Bernoulli polynomials of order $k$.

\noindent
For example, if setting  $0 \leq n \leq 3, \; 0 \leq k \leq n$ , in (\ref{24}), we have
\begin{eqnarray*}
\left(
\begin{array}{cccc}
1 & 0 & 0 & 0 \\
0 & 1 & 0 & 0 \\
0 & -1 & 1 & 0  \\
0 & 2 & -3 & 1
\end{array}
\right)\left(
\begin{array}{cccc}
1 & 1 & 1 & 1 \\
x & x-\frac{1}{2} & x-1 & x-\frac{3}{2} \\
x^2 & x^2-x+\frac{1}{6} & x^2-2x+\frac{5}{6} & x^2-3x+2 \\
x^3 & x^3-\frac{3}{2}x^2+\frac{1}{2} x & x^3-3x^2+\frac{5}{2} x-\frac{1}{2} & x^3-\frac{9}{2}x^2+6x-\frac{9}{4}
\end{array}
\right)=
\\
\; \; \; \; \; \; \; \;
\left(
\begin{array}{cccc}
1 & 1 & 1 & 1 \\
x & x-\frac{1}{2} & x-1 & x-\frac{3}{2}  \\
x^2-x & x^2-2x+\frac{2}{3} & x^2-3x+\frac{11}{6} & x^2-4x+\frac{7}{2} \\
x^3-3x^2+2x & x^3-\frac{9}{2}x^2+\frac{11}{2}x-\frac{3}{2} & x^3-6x^2+\frac{21}{2}x-5 & x^3-\frac{15}{2}x^2 + 17x-\frac{45}{4}\\
\end{array}
\right)
\end{eqnarray*}

\noindent
Kim et al. \cite[Theorem 7]{Kimetal2014} introduced the following result. For $n\in \mathbb{Z}, \; k\in \mathbb{N}$,
\begin{equation} \label{25}
B_n^{(k) } (x) = \sum_{m=0}^n D_m^{(k) } (x) s_2 (n,m).
\end{equation}

\noindent
We can write Eq. (\ref{25}) in the matrix form as follows
\begin{equation}\label{26}
{\bf B}^{(k) } (x)={\bf S}_2 \, {\bf D}^{(k) } (x).
\end{equation}

\noindent
For example, if setting  $0≤\leq n \leq 3, \; 0 \leq k \leq n$ , in (\ref{26}), we have
{\footnotesize
\begin{eqnarray*}
\left(
\begin{array}{cccc}
1 & 0 & 0 & 0 \\
0 & 1 & 0 & 0 \\
0 & 1 & 1 & 0 \\
0 & 1 & 3 & 1
\end{array}
\right)
\left(
\begin{array}{cccc}
1 & 1 & 1 & 1 \\
x & x-\frac{1}{2} & x-1 & x- \frac{3}{2} \\
x^2-x & x^2-2x+\frac{2}{3} & x^2-3x+\frac{11}{6} & x^2-4x+\frac{7}{2} \\
x^3-3x^2+2x & x^3-\frac{9}{2} x^2+\frac{11}{2}x-\frac{3}{2} & x^3-6x^2+\frac{21}{2}x-5 & x^3-\frac{15}{2}x^2+17x-\frac{45}{4}
\end{array}
\right)
\\
= \left(
\begin{array}{cccc}
1 & 1 & 1 & 1 \\
x & x-\frac{1}{2} & x-1 & x-\frac{3}{2} \\
x^2 & x^2-x + \frac{1}{6} & x^2-2x+\frac{5}{6} &x^2-3x+2\\
x^3 & x^3-\frac{3}{2}x^2+\frac{1}{2} x & x^3-3x^2+\frac{5}{2}x-\frac{1}{2} & x^3-\frac{9}{2}x^2+6x-\frac{9}{4}
\end{array}
\right).
\end{eqnarray*}
}
\begin{remark} We can prove Theorem 7 in Kim et al. \cite{Kimetal2014} by using the matrix form (\ref{24}) as follows . Multiplying both sides of (\ref{24}) by the Striling number of second kind, we have
\begin{equation*}
{\bf S}_2 \, {\bf D}^{(k)} (x) ={\bf S}_2 \, {\bf S}_1 \,  {\bf B}^{(k)} (x)={\bf I\, B}^{(k)} (x)={\bf B}^{(k)} (x).
\end{equation*}
\end{remark}

\noindent
Kim et al.  \cite{Kimetal2014} defined the Daehee numbers of the second kind of order $k$ by the generating function as follows.
\begin{equation} \label{27}
\sum_{n=0}^\infty \hat{D}_n^k [(k)] \frac{ t^n}{n!}=\left(\frac{(1-t)  \log⁡(1-t)}{-t} \right)^k.
\end{equation}

\noindent
Kim et al.  \cite[Theorem 8]{Kimetal2014} introduced the following result. For  $n\in \mathbb{Z}, \; k \in \mathbb{N}$,
\begin{equation} \label{28}
\hat{D}_n^k [(k)]=\sum_{l=0}^n \Big[^{\; n}_{\; l}\Big] B_l^{(k) },
\end{equation}

\noindent
where $\Big[^{\; n}_{\; l} \Big]=(-1)^{n-l} s_1 (n,l)=|s_1(n,k)|=\mathfrak{s}(n,k)$, where $\mathfrak{s}(n,k)$ is the signless Stirling numbers of the first kind, see \cite{Comtet1974} and \cite{El-Desouky1994, El-Desoukyetal2010}.

\noindent
We can write this theorem in the matrix form as follows
\begin{equation}\label{29}
\hat{{\bf D}}^{(k)}= {\bf \mathfrak{S}} \, {\bf B}^{(k) },
\end{equation}

\noindent
where $\hat{{\bf D}}^{(k) }$  is $(n+1)\times(k+1)$  matrix of Daehee numbers of the second kind with order $k$ and ${\bf \mathfrak{S}}$ is $(n+1)\times(n+1)$ lower triangular matrix for the signless Stirling numbers of first kind.

\noindent
For example, if setting $0 \leq n \leq 3,\; 0 \leq k \leq n$ in (\ref{29}), we have
\begin{equation*}
\left(
\begin{array}{cccc}
1 & 0 & 0 & 0 \\
0 & 1 & 0 & 0 \\
0 & 1 & 1 & 0 \\
0 & 2 & 3 & 1
\end{array}
\right)\left(
\begin{array}{cccc}
1 & 1 & 1 & 1 \\
0 & -1/2 & -1 & -3/2 \\
0 & 1/6 & 5/6 & 2 \\
0 & 0 & -1/2 & -9/4
\end{array}
\right)=\left(
\begin{array}{cccc}
1	& 1	& 1 &	1\\
0	& -1/2	& -1	& -3/2\\
0	& -1/3	& -1/6	& 1/2\\
0	& -1/2	& 0	& 3/4
\end{array}
\right).
\end{equation*}

\noindent
Kim et al.  \cite[Theorem 9]{Kimetal2014} introduced the following result.
For  $n\in \mathbb{Z},\; k\in \mathbb{N}$, we have
\begin{equation} \label{30}
B_n^{(k)}= \sum_{m=0}^n (-1)^{n-m} s_2 (n,m)  \hat{D}_m^k [(k)].
\end{equation}

\noindent
We can write Eq. (\ref{30}) in the matrix form as follows
\begin{equation} \label{31}
{\bf B}^{(k)}=  \tilde{\bf S}_2 \, \hat{{\bf D}}^{(k)}.
\end{equation}

\noindent
where $\tilde{\bf S}_2$ is $(n+1)\times(n+1)$ lower triangular matrix for signed Stirling numbers of the second kind defined by
\begin{equation*}
({\tilde{\bf S}}_2)_{ij}=
\left\{
\begin{array}{cl}
(-1)^{i-j} s_2(i,j), & i\geq j, \\
0, & \mbox{otherwise}.
\end{array}
\right.,
\; i,j=0,1,\cdots,n.
\end{equation*}

\noindent
For example, if setting  $0\leq n \leq 3,\; 0 \leq k \leq n$ in (\ref{31}), we have
\begin{equation*}
\left(
\begin{array}{cccc}
1 & 0 & 0 & 0 \\
0 & 1 & 0 & 0 \\
0 & -1 & 1 & 0 \\
0 & 1 & -3 & 1
\end{array}
\right)
\left(
\begin{array}{cccc}
1	& 1	& 1 &	1\\
0	& -1/2	& -1	& -3/2\\
0	& -1/3	& -1/6	& 1/2\\
0	& -1/2	& 0	& 3/4
\end{array}
\right)=
\left(
\begin{array}{cccc}
1 & 1 & 1 & 1 \\
0 & -1/2 & -1 & - 3/2 \\
0 & 1/6 & 5/6 & 2 \\
0 & 0 & - 1/2 & - 9/4
\end{array}
\right).
\end{equation*}

\noindent
\begin{remark}
We can prove Theorem 9 in Kim et al. \cite{Kimetal2014} by using the matrix form (\ref{29}) as follows.\\
Multiplying both sides of (\ref{29}) by the matrix of sign Striling numbers of second kind $\tilde{\bf S}_2$ we have
\begin{equation*}
\tilde{\bf S}_2 \, \hat{{\bf D}}^{(k) }=\tilde{\bf S}_2 \, {\bf \mathfrak{S}} \, {\bf B}^{(k)}={\bf I} \, {\bf B}^{(k)}= {\bf B}^{(k)},
\end{equation*}

\noindent
we obtain Eq. (\ref{31}), where we used the identity,  $\tilde{\bf S}_2 \, {\bf \mathfrak{S}}={\bf I}$.
\end{remark}

\noindent
Kim et al.  \cite{Kimetal2014} defined the Daehee polynomials of the second kind of order $k$ by the generating function as follows.
\begin{equation} \label{32}
\sum_{n=0}^\infty \hat{D}_n^k [(k)] (x)  \frac{t^n}{n!}=\left(\frac{(1-t)  \log⁡(1-t)}{-t} \right)^k (1-t)^x.
\end{equation}

\noindent
Kim et al.  \cite[Eq. (31)]{Kimetal2014} introduced the following result. For  $n\in \mathbb{Z},\; k\in \mathbb{N}$,
\begin{equation} \label{33a}
\hat{D}_n^k [(k)](x)= \sum_{m=0}^n (-1)^{n-m} s_1 (n,m) B_m^{(k) } (-x).
\end{equation}

\noindent
Eq. (\ref{33a}) is equivalent to

\begin{equation} \label{33}
\hat{D}_n^k [(k)](x)= \sum_{m=0}^n  \mathfrak{s} (n,m) B_m^{(k) } (-x).
\end{equation}

\noindent
We can write Eq. (\ref{33}) in the matrix form as follows
\begin{equation} \label{34}
\hat{{\bf D}}^{(k) } (x) = {\bf \mathfrak{S}} \,  {\bf B}^{(k)} (-x),
\end{equation}

\noindent
where  $\hat{{\bf D}}^{(k) } (x) $ is $(n+1)\times(k+1)$ matrix of the Daehee polynomials of the second kind with order $k$ and ${\bf B}^{(k) } (x)$ is $(n+1)\times(k+1)$, the matrix of the Bernoulli polynomials when $x\rightarrow -x $ numbers.

\noindent
For example, if setting  $0 \leq n \leq 3,\; 0 \leq k \leq n,$ in (\ref{34}), we have
\begin{eqnarray*}
\left(
\begin{array}{cccc}
1 & 0 & 0 & 0 \\
0 & 1 & 0 & 0 \\
0 & 1 & 1 & 0 \\
0 & 2 & 3 & 1
\end{array}
\right)
\left(
\begin{array}{cccc}
1 & 1 & 1 & 1 \\
-x & -x-\frac{1}{2} & -x-1 & -x-\frac{3}{2} \\
x^2 & x^2+x+\frac{1}{6} & x^2+2x+\frac{5}{6} & x^2+3x+2 \\
-x^3 & -x^3-\frac{3}{2}x^2-\frac{1}{2}x & -x^3-3x^2-\frac{5}{2}x-\frac{1}{2} & -x^3-\frac{9}{2}x^2-6x-\frac{9}{4}
\end{array}
\right)
\\
=
\left(
\begin{array}{cccc}
1 &	1 &	1 &	1 \\
-x	& -x-\frac{1}{2}	& -x-1 &	-x-\frac{3}{2} \\
x^2-x &	x^2-\frac{1}{3} &	x^2+x-\frac{1}{6} &	x^2+2x+\frac{1}{2} \\
3x^2-x^3-2x &	\frac{3x^2}{2}-x^3+\frac{x}{2}-\frac{1}{2} &	\frac{3x}{2}-x^3 &	x-\frac{3x^2}{2}-x^3+ \frac{3}{4}
\end{array}
\right).
\end{eqnarray*}

\noindent
Kim et al.  \cite[Theorem 11]{Kimetal2014} introduced the following result. For  $n\in \mathbb{Z}, \;k\in \mathbb{N}$,
\begin{equation} \label{35}
B_n^{(k)} (-x) = \sum_{m=0}^n (-1)^{n-m} s_2 (n,m) \hat{D}_m^k [(k)](x).
\end{equation}

\noindent
We can write Eq. (\ref{35}) in the matrix form as follows
\begin{equation}\label{36}
{\bf B}^{(k) } (-x)= \tilde{\bf S}_2 \, \hat{{\bf D}}^{(k)} (x).
\end{equation}

\noindent
For example, if setting  $0\leq n \leq 3, \; 0 \leq k \leq n$ in (\ref{36}), we have
\begin{eqnarray*}
\left(
\begin{array}{cccc}
1	& 0	 & 0  & 0 \\
0	& 1	 & 0  & 0  \\
0	& -1 & 1  & 0 \\
0	& 1	 & -3 & 1 \\
\end{array}
\right)
\left(
\begin{array}{cccc}
1 &	1 &	1 &	1 \\
-x	& -x-\frac{1}{2}	& -x-1 &	-x-\frac{3}{2} \\
x^2-x &	x^2-\frac{1}{3} &	x^2+x-\frac{1}{6} &	x^2+2x+\frac{1}{2} \\
3x^2-x^3-2x &	\frac{3x^2}{2}-x^3+\frac{x}{2}-\frac{1}{2} &	\frac{3x}{2}-x^3 &	x-\frac{3x^2}{2}-x^3+ \frac{3}{4}
\end{array}
\right)
\\
= \left(
\begin{array}{cccc}
1 & 1 & 1 & 1 \\
-x & -x-\frac{1}{2} & -x-1 & -x-\frac{3}{2} \\
x^2 & x^2+x+\frac{1}{6} & x^2+2x+\frac{5}{6} & x^2+3x+2 \\
-x^3 & -x^3-\frac{3x^2}{2}-\frac{x}{2} & -x^3-3x^2-\frac{5x}{2}-\frac{1}{2} & -x^3-\frac{9x^2}{2}-6x-\frac{9}{4}
\end{array}
\right).
\end{eqnarray*}

\begin{remark}
We can prove Eq. (\ref{36}), \cite[Theorem 11]{ Kimetal2014}, directly by using the matrix form (\ref{34}) as follows. Multiplying both sides of (\ref{34}) by  $\tilde{\bf S}_2$ as follows.
\begin{equation*}
\tilde{\bf S}_2 \, \hat{{\bf D}}^{(k)} (x)=\tilde{\bf S}_2 \, {\bf \mathfrak{S}} \, {\bf B}^{(k)} (-x)={\bf I} \, {\bf B}^{(k)} (-x)={\bf B}^{(k)} (-x),
\end{equation*}
thus, we have Eq. (\ref{36}).
\end{remark}

\section{The $\lambda$- Daehee Numbers and Polynomials of Higher Order}
In this section we introduce the matrix representation for the $\lambda$-Daehee numbers and polynomials of higher order given by  Kim  et al. \cite{Kimetal2013}. Hence, we can derive these results in matrix representation and prove these results simply by using the given matrix forms.

\noindent
The $\lambda$-Daehee polynomials of the first kind with order $k$ can be defined by the generating function
\begin{equation} \label{37}
\left(\frac{\lambda \log⁡(1+t)}{(1+t)^{\lambda}-1} \right)^k  (1+t)^x = \sum_{n=0}^\infty D_{n,\lambda}^{(k)} (x) \frac{ t^n}{n!}.
\end{equation}

\noindent
When $x=0,\; D_{n,\lambda}^{(k)}=D_{n, \lambda}^{(k)} (0)$ are called the $\lambda$-Daehee numbers of order $k$.
\begin{equation} \label{ln}
\left(\frac{\lambda  \log⁡(1+t)}{(1+t)^{\lambda}-1} \right)^k  = \sum_{n=0}^\infty D_{n,\lambda}^{(k)} \frac{t^n}{n!}.
\end{equation}

\noindent
It is easy to see that $D_n^{(k) } (x) = D_{n,1}^{(k) } (x)$ and $D_{n,\lambda} (x)=D_{n,\lambda}^{(1) } (x)$.

\noindent
Kim et al. \cite[Theorem 3]{Kimetal2013} obtained the following results. For  $n\geq 0, \;k \in \mathbb{N}$,
\begin{equation} \label{38}
D_{n,\lambda}^{(k)} (x)= \sum_{m=0}^n s_1 (n,m) \lambda^m B_m^{(k)} \left(\frac{x}{\lambda} \right),
\end{equation}
and
\begin{equation} \label{39}
\lambda^n B_n^{(k)} \left(\frac{x}{\lambda}\right)=\sum_{m=0}^n s_2 (n,m) D_{m,\lambda}^{(k) } (x),
\end{equation}

\noindent
we can write these results in the following matrix forms
\begin{equation} \label{40}
{\bf D}_{\lambda}^{(k)} (x)= {\bf S}_1 \, {\bf \Lambda \,  B}^{(k) } \left(\frac{x}{\lambda} \right),
\end{equation}
and
\begin{equation} \label{41}
 {\bf \Lambda}\,  {\bf B}^{(k)} \left(\frac{x}{\lambda} \right)= {\bf S}_2\,  {\bf D}_{\lambda}^{(k) } (x),
\end{equation}

\noindent
where, ${\bf D}_{\lambda}^{(k)} (x)$ is $(n+1)\times (k+1)$ matrix for the $\lambda$-Daehee polynomials of the first kind with order $k$, ${\bf B}^{(k)} (x/\lambda)$ is $(n+1)\times(k+1)$ matrix for the Bernoulli polynomials of order $k$, when $x\rightarrow x/\lambda$ and ${\bf \Lambda}$ is $(n+1)\times(n+1)$ diagonal matrix with elements, $({\bf \Lambda})_{ii}=\lambda^i,\;    i=j=0,1,\cdots,n$.

\noindent
For example, if setting  $0 \leq n \leq 3,\; 0\leq k \leq n$, in (\ref{40}), we have
{\footnotesize
\begin{eqnarray*}
&&
\left(
\begin{array}{cccc}
1 & 0 & 0 & 0 \\
0 & 1 & 0 & 0 \\
0 & -1 & 1 & 0 \\
0 & 2 & -3 & 1
\end{array}
\right)
\left(
\begin{array}{cccc}
1 & 0 & 0 & 0 \\
0 & \lambda & 0 & 0 \\
0 & 0 & \lambda^2 & 0 \\
0 & 0 & 0 & \lambda^3
\end{array}
\right) \times
\\
&&
\left(
\begin{array}{cccc}
1 & 1 & 1 & 1 \\
\frac{x}{\lambda} & \frac{x}{\lambda}-\frac{1}{2} & \frac{x}{\lambda}-1 & \frac{x}{\lambda}-\frac{3}{2} \\
\frac{x^2}{\lambda^2} & \frac{x^2}{\lambda^2} -\frac{x}{\lambda} + \frac{1}{6} & \frac{x^2}{\lambda^2} -\frac{2x}{\lambda} + \frac{5}{6} & \frac{x^2}{\lambda^2} -\frac{3 x}{\lambda} + 2 \\
\frac{x^3}{\lambda^3} & \frac{x}{2\lambda}-\frac{3x^2}{2\lambda^2} + \frac{x^3}{\lambda^3} & \frac{5x}{2\lambda}- \frac{3x^2}{\lambda^2} + \frac{x^3}{\lambda^3} -\frac{1}{2} & \frac{6x}{\lambda}-\frac{9x^2}{2\lambda^2} + \frac{x^3}{\lambda^3} -\frac{9}{4}
\end{array}
\right)=
\\
&&
\left(
\begin{array}{cccc}
1 & 1 & 1 & 1\\
x & x-\frac{\lambda}{2} & x-\lambda & x-\frac{3\lambda}{2} \\
x(x-1)& \frac{\lambda^2}{6} - \lambda x+  \frac{\lambda}{2} + x^2-x & \frac{5\lambda^2}{6}- 2\lambda x + \lambda + x^2-x & 2\lambda^2 - 3 \lambda x + \frac{3\lambda}{2} + x^2-x \\
D_{3,\lambda}^{(0)}(x)& D_{3,\lambda}^{(1)}(x) & D_{3,\lambda}^{(2)}( x) & D_{3,\lambda}^{(3)}(x)
\end{array}
\right).
\end{eqnarray*}
}
where
\begin{eqnarray*}
D_{3,\lambda}^{(0)}(x)& = &x(x-1)(x-2), \;
D_{3,\lambda}^{(1)}(x)=- \frac{1}{2}(\lambda -2x+2)(\lambda-2x + x^2-\lambda x),
\\
D_{3,\lambda}^{(2)}(x) & = &-\frac{1}{2} (\lambda -x+1)(\lambda^2- 4\lambda x + 4\lambda + 2x^2-4x),
\\
D_{3,\lambda}^{(3)}(x)& = & -\frac{1}{4} (3\lambda-2x+2)(3\lambda^2-6\lambda x+6 \lambda + 2x^2-4x).
\end{eqnarray*}

\begin{remark} In fact, we can prove Eq. (\ref{41}), simply by multiplying Eq. (\ref{40}) by ${\bf S}_2$ as follows.
\begin{equation*}
{\bf S}_2 \, {\bf D}_{\lambda}^{(k)} (x)={\bf S}_2 \, {\bf S}_1 \,  {\bf \Lambda B}^{(k)} \left(\frac{x}{\lambda} \right)={\bf I \, \Lambda B}^{(k) } \left(\frac{x}{\lambda} \right)={\bf \Lambda \,  B}^{(k)} \left(\frac{x}{\lambda} \right).
\end{equation*}
\end{remark}

\noindent
The following Theorem gives the relation between the Daehee polynomials of higher order and $\lambda$-Daehee polynomials of higher order.
\begin{theorem}
For $m \geq 0$, we have
\begin{equation} \label{42}
D_{m,\lambda}^{(k) } (\lambda x)=m! \sum_{n=0}^m \sum_{ i_1+i_2+ \cdots+i_n=m} \frac{D_n^{(k)}(x)}{n!} \left(^{\lambda}_{i_1} \right) \left(^{\lambda}_{i_2}\right) \cdots \left(^{\lambda}_{i_n} \right).
\end{equation}
\end{theorem}
\begin{proof}
From (\ref{14}), replacing $(1+t)$ by $(1+t)^\lambda$, we have
\begin{equation*}
\left(\frac{\lambda \log⁡(1+t)}{(1+t)^\lambda-1}\right)^k (1+t)^{\lambda x}= \sum_{n=0}^\infty D_n^{(k)} (x)  \frac{((1+t)^\lambda-1 )^n}{n!},
\end{equation*}
thus from (\ref{37}), we get
\begin{eqnarray*}
\sum_{m=0}^\infty D_{m,\lambda}^{(k)} (\lambda x) \frac{ t^m}{m!} &= & \sum_{n=0}^\infty \frac{D_n^{(k)} (x)}{n!}  \left(
\sum_{i=0}^\lambda \left(^{\lambda}_i \right) t^i -1 \right)^n
\\
&= &
\sum_{n=0}^\infty \frac{D_n^{(k)} (x)}{n!}  \left(
\sum_{i=1}^\lambda \left(^{\lambda}_i \right) t^i \right)^n.
\end{eqnarray*}

\noindent
Using Cauchy rule of product of series, we obtain
\begin{eqnarray*}
\sum_{m=0}^\infty D_{m,\lambda}^{(k)} (\lambda x)   \frac{t^m}{m!} & = &
\sum_{n=0}^\infty \frac{D_n^{(k)} (x))}{n!} \sum_{m=n}^\infty \sum_{i_1+i_2+ \cdots +i_n=m} \left(^{\lambda}_{i_1} \right) \cdots \left(^{\lambda}_{i_n} \right) t^m
\\
&= &
\sum_{m=0}^\infty m! \sum_{n=0}^m  \sum_{i_1+i_2+ \cdots +i_n=m}
\frac{D_n^{(k)} (x))}{n!}  \left(^{\lambda}_{i_1} \right) \cdots \left(^{\lambda}_{i_n} \right) \frac{t^m}{m!}.
\end{eqnarray*}

\noindent
Equating the coefficients of $t^m$ on both sides yields (\ref{42}). This completes the proof.
\end{proof}

\noindent
Setting  $x=0$, in (\ref{42}), we have the following corollary as a special case.
\begin{corollary}
For $m \geq 0$, we have
\begin{equation} \label{43}
D_{m,\lambda}^{(k)} = m! \sum_{n=0}^m \sum_{i_1+i_2+ \cdots+i_n=m} \frac{D_n^{(k) }}{n!} \left(^{\lambda}_{i_1} \right)\left(^{\lambda}_{i_2} \right) \cdots \left(^{\lambda}_{i_n} \right).
\end{equation}
\end{corollary}

\noindent
Kim et al. \cite{Kimetal2013}, defined the $\lambda$-Daehee polynomials of the second kind with order $k$ as follows.
\begin{equation} \label{44}
\left( \frac{\lambda \log⁡(1+t)}{(1+t)^\lambda-1} \right)^k  (1+t)^{\lambda k+x}) = \sum_{n=0}^\infty \hat{D}_{n,\lambda}^{(k)} (x)\frac{ t^n}{n!}.
\end{equation}

\noindent
Kim et al. \cite[Theorem 5]{Kimetal2013} proved that
\begin{equation} \label{45}
\hat{D}_{m, \lambda}^{(k)} (x) = \sum_{l=0}^m s_1 (m,l)  \lambda^l B_l^{(k)} \left(k+\frac{x}{\lambda} \right),
\end{equation}
and
\begin{equation} \label{46}
\lambda^m B_m^{(k) } \left(k+\frac{x}{\lambda}\right)= \sum_{n=0}^m s_2 (m,n) \hat{D}_{n,\lambda}^{(k)} (x).
\end{equation}

\noindent
Also, Kim et al. \cite[Eq. (35)]{Kimetal2013} introduced the following result
\begin{equation} \label{47}
B_n^{(k) } (k-x)=(-1)^n B_n^{(k)}(x).
\end{equation}

\noindent
\begin{remark}
We can write (\ref{45}) and (\ref{46}), respectively, in the following matrix forms
\begin{equation} \label{48}
\hat{\bf D}_{\lambda}^{(k)} (x)={\bf S}_1 \,  {\bf \Lambda}_1 \,  {\bf B}^{(k)} \left(-\frac{x}{\lambda}\right),
\end{equation}
and
\begin{equation} \label{49}
{\bf \Lambda}_1 \,  {\bf B}^{(k)} \left(-\frac{x}{\lambda}\right)=  {\bf S}_2 \,  \hat{\bf D}_{\lambda}^{(k)} (x),
\end{equation}
\end{remark}

\noindent
where $\hat{\bf D}_\lambda (x)$   is $(n+1)\times(n+1)$ matrix for the $\lambda$-Daehee polynomials of the second kind of order $k$  and ${\bf \Lambda}_1$ is $(n+1)\times(n+1)$ diagonal matrix with elements $({\bf \Lambda}_1 )_{ii}=(-\lambda)^i$, for  $i=j=0,1,\cdots ,n$.

\noindent
For example, if  setting $0\leq n \leq 3, \;0 \leq k \leq n$, in (\ref{48}), we have
{\footnotesize
\begin{eqnarray*}
& &
\left(
\begin{array}{cccc}
1 & 0 & 0 & 0 \\
0 & 1 & 0 & 0 \\
0 & -1 & 1 & 0 \\
0 & 2 & -3 & 1
\end{array}
\right)\left(
\begin{array}{cccc}
1 & 0 & 0 & 0 \\
0 & -\lambda & 0 & 0 \\
0 & 0 & \lambda^2 & 0 \\
0 & 0 & 0 & -\lambda^3
\end{array}
\right)
\times
\\
& &
\left(
\begin{array}{cccc}
1 & 1 & 1 & 1 \\
-\frac{x}{\lambda} & -\frac{x}{\lambda}-\frac{1}{2} &- \frac{x}{\lambda}-1 &-\frac{x}{\lambda}-\frac{3}{2} \\
\frac{x^2}{\lambda^2} & \frac{1}{\lambda^2}(\frac{\lambda^2}{6}+\lambda x+x^2 )& \frac{1}{\lambda^2}( \frac{5}{6} \lambda^2+2\lambda x+x^2 ) & \frac{1}{\lambda^2}(2\lambda + x)(\lambda + x) \\
-\frac{x^3}{\lambda^3} & \frac{x(\lambda + 2x)( \lambda + x)}{2\lambda^3 }  &-\frac{(\lambda + x)(\lambda^2+4\lambda x+2x^2 )}{2\lambda^3} &-\frac{(3\lambda +2x)(3\lambda^2+6\lambda x+2x^2 )}{4\lambda^3}
\end{array}
\right)=
\\
& &
\left(
\begin{array}{cccc}
1 & 1 & 1 & 1\\
x & \frac{\lambda}{2}+x & \lambda+x &\frac{3\lambda}{2} + x \\
x(x-1) & \frac{\lambda^2 }{6}+\lambda x-\frac{\lambda}{2}+x^2-x & \frac{5\lambda^2}{6} +2\lambda x-\lambda+x^2-x &2\lambda^2+3\lambda x-\frac{3\lambda}{2} +x^2-x \\
\hat{D}_{3,\lambda}^{(0)}&\hat{D}_{3,\lambda}^{(1)}(x)& \hat{D}_{3,\lambda}^{(2)}(x) & \hat{D}_{3,\lambda}^{(3)}(x)
\end{array}
\right).
\end{eqnarray*}
}

where
\begin{eqnarray*}
\hat{D}_{3,\lambda}^{(0)}(x) & =& x(x-1)(x-2), \;  \hat{D}_{3,\lambda}^{(1)}(x)=-\frac{1}{2}(\lambda+2x-2)(\lambda+2x-x^2-\lambda x)
\\
\hat{D}_{3,\lambda}^{(2)}(x)& = & \frac{1}{2} (\lambda + x-1)(\lambda^2 + 4\lambda x-4\lambda + 2x^2-4x),
\\
\hat{D}_{3,\lambda}^{(3)}(x)&=& \frac{1}{4} (3\lambda+2x-2)(3\lambda^2+6\lambda x-6\lambda+2x^2-4x))).
\end{eqnarray*}

\begin{remark}
We can prove Eq. (\ref{46}) easily by using the matrix form, multiplying Eq.(\ref{48}) by ${\bf S}_2$ as follows.
\begin{equation*}
{\bf S}_2 \,  \hat{\bf D}_\lambda^{(k)} (x)={ \bf S}_2 \,  {\bf S}_1 \,  { \bf \Lambda}_1 \,  {\bf B}^{(k)} \left(-\frac{x}{\lambda} \right)={\bf I} \, {\bf \Lambda}_1 \,  {\bf B}^{(k)} \left(- \frac{x}{\lambda} \right)={\bf \Lambda}_1 \,  {\bf B}^{(k)} \left(- \frac{x}{\lambda} \right).
\end{equation*}

\end{remark}

\section{The Twisted $\lambda$-Daehee Numbers and Polynomials of Higher Order}
Kim et al. \cite{Kimetal2013b} defined the twisted $\lambda$-Daehee polynomials of the first kind of order $k$ by the generating function
\begin{equation} \label{50}
\left(\frac{\lambda  \log⁡(1+\xi t)}{(1+\xi t)^\lambda-1} \right)^k (1+\xi t)^x = \sum_{n=0}^\infty D_{n,\xi }^{(k)} (x|\lambda ) \frac{ t^n}{n!}.
\end{equation}

\noindent
In the special case, $x=0,\; D_{n,\xi, \lambda}^{(k) } = D_{n,\xi }^{(k)} (0|\lambda)$ are called the twisted $\lambda$-Daehee numbers of the first kind of order $k$.
\begin{equation} \label{51}
\left( \frac{ \lambda \log⁡(1+\xi t)}{(1+\xi t)^\lambda-1} \right)^k = \sum_{n=0}^\infty D_{n,\xi,\lambda}^{(k)} \frac{ t^n}{n!}.
\end{equation}

\noindent
The twisted Bernoulli polynomials of order $r\in \mathbb{N}$  are defined by the generating function, see \cite{Dolgyetal2013}
\begin{equation} \label{52}
\left(\frac{t}{\xi e^t-1} \right)^r e^{xt}=\sum_{n=0}^\infty B_{n,\xi}^{(r)} (x)  \frac{t^n}{n!}.
\end{equation}

\noindent
The relation between the twisted $\lambda$-Daehee polynomials and $\lambda$-Daehee polynomials of order $k$, can be obtained by the following corollary.
\begin{corollary}
For $n \geq 0, \, k\in \mathbb{N}$, we have
\begin{equation} \label{a1}
D_{n,\xi}^{(k) } (x|\lambda)=\xi^n  D_{n,\lambda}^{(k) } (x).
\end{equation}
\end{corollary}
\begin{proof}
Replacing  $t$  with $\xi t$ in (\ref{37}), we have
\begin{equation} \label{a2}
\left( \frac{\lambda \log⁡(1+\xi t)}{(1+\xi t)^\lambda-1} \right)^k  (1+\xi t)^x= \sum_{n=0}^\infty D_{n,\lambda}^{(k)} (x)
\frac{(\xi t)^n}{n!} = \sum_{n=0}^\infty \xi^n  D_{n,\lambda}^{(k) } (x)  \frac{t^n}{n!},
\end{equation}

\noindent
hence by virtue of (\ref{50}) and (\ref{a2}), we get (\ref{a1}). This completes the proof.
\end{proof}

\noindent
If we put in (\ref{a1}), $x=0$ and $\lambda=1$, respectively,  we have
\begin{equation*}
D_{n,\xi,\lambda}^{(k)}=\xi^n  D_{n,\lambda}^{(k)}.
\end{equation*}
and
\begin{equation*}
D_{n,\xi}^{(k)} (x)=\xi^n  D_n^{(k)}(x).
\end{equation*}

\noindent
Kim et al. \cite[Theorem 1]{Kimetal2013b} proved the following relation. For $m\in \mathbb{Z}, \; k\in \mathbb{N}$,
\begin{equation} \label{53}
D_{m,ξ}^{(k)} (x│ \lambda )=\xi^m \sum_{l=0}^m S_1 (m,l) \lambda^l B_l^{(k) } \left(\frac{x}{\lambda} \right),
\end{equation}
and
\begin{equation} \label{54}
\lambda^m B_{m,\xi^\lambda}^{(k)} \left(\frac{x}{\lambda} \right)=\sum_{n=0}^m D_{n,\xi}^{(k)} (x│ \lambda ) \xi^{-n-x} s_2 (m,n).
\end{equation}

\noindent
where $B_{m,\xi^\lambda }^{(k) } \left(\frac{x}{\lambda}\right)$ is defined by Kim et al. \cite[Eq. 15]{Kimetal2013b}, as follows
\begin{equation} \label{55}
\left(\frac{\lambda t}{\xi^\lambda e^{\lambda t}-1}\right)^k (\xi e^t )^x = \xi^x \sum_{m=0}^\infty \lambda^m B_{m,\xi^\lambda}^{(k)} \left(\frac{x}{\lambda}\right)  \frac{ t^m}{m!}.
\end{equation}

\begin{remark}
We can write (\ref{53}) in the following matrix form
\begin{equation} \label{56}
{\bf D}_\xi^{(k)} (x|\lambda)={\bf \Xi S}_1 \,  {\bf \Lambda B}^{(k)} \left(\frac{x}{\lambda}\right),
\end{equation}

\end{remark}

\noindent
where ${\bf D}_\xi^{(k)} (x|\lambda)$ is  $(n+1)\times(k+1)$ matrix for the twisted Daehee numbers of the first kind of the order $k$ and ${\bf \Xi}$ is $(n+1)\times(n+1)$ diagonal matrix with elements $({\bf \Xi})_{ii}=\xi^i$  for $i=j=0,1,\cdots,n$.

\noindent
For example, if setting $0\leq n \leq 3,\; 0\leq k \leq n$, in (\ref{56}), we have
{\footnotesize
\begin{eqnarray*}
 \left(
\begin{array}{cccc}
1 & 0 & 0 & 0\\
0 & \xi & 0 & 0\\
0 & 0 & \xi^2 & 0\\
0 & 0 & 0 & \xi^3
\end{array}
\right)
\left(
\begin{array}{cccc}
1 & 0 & 0 & 0 \\
0 & 1 & 0 & 0 \\
0 & -1 & 1 & 0 \\
0 & 2 & -3 &1
\end{array}
\right)
\left(
\begin{array}{cccc}
1 & 0 & 0 & 0 \\
0 & \lambda & 0 & 0 \\
0 & 0 & \lambda^2 & 0 \\
0 & 0 & 0 & \lambda^3
\end{array}
\right)\times
\; \; \; \; \; \; \; \; \; \; \; \; \; \; \; \; \; \; \; \; \; \; \; \; \; \; \; \; \; \; \; \;
\; \; \; \; \; \; \; \; \; \; \; \; \; \; \; \; \; \; \; \; \; \; \; \; \; \; \; \; \; \; \; \;
\\
\left(
\begin{array}{cccc}
1 & 1 & 1 & 1\\
\frac{x}{\lambda} & \frac{x}{\lambda}-\frac{1}{2} & \frac{x}{\lambda}-1 & \frac{x}{\lambda}-\frac{3}{2} \\
\frac{x^2}{\lambda^2} & \frac{x^2}{\lambda^2} -\frac{x}{\lambda}+\frac{1}{6} & \frac{x^2}{\lambda^2} -\frac{2x}{\lambda}+ \frac{5}{6} & \frac{x^2}{\lambda^2} -\frac{3x}{\lambda}+2 \\
\frac{x^3}{\lambda^3} & \frac{x}{2\lambda}-\frac{3x^2}{2\lambda^2}+ \frac{x^3}{\lambda^3} & \frac{5x}{2\lambda}-\frac{3x^2}{\lambda^2} + \frac{x^3}{\lambda^3} -\frac{1}{2} & \frac{6x}{\lambda}-\frac{9x^2}{2\lambda^2}+ \frac{x^3}{\lambda^3} -\frac{9}{4}
\end{array}
\right)=
\; \; \; \; \; \; \; \; \; \; \; \; \; \; \; \; \; \; \; \; \; \; \; \; \; \; \; \; \; \; \; \;
\; \; \; \; \; \; \; \; \; \; \; \; \; \; \; \; \; \; \; \; \; \; \; \; \;
\\
\left(
\begin{array}{cccc}
1 & 1 & 1 & 1 \\
\xi x & -\frac{\xi}{2} (\lambda-2x)&-\xi(\lambda-x)&-\frac{\xi}{2}(3\xi-2x) \\
\xi^2 x(x-1) & \xi^2 (\frac{\lambda^2}{6}-\lambda x+ \frac{\lambda}{2}+x^2-x)& \xi^2 (\frac{5}{6} \lambda^2-2\lambda x+\lambda+x^2-x)& \xi^2 (2\lambda^2-3\lambda x+\frac{3}{2} \lambda+x^2-x)\\
D_{3,\xi}^{(0)}(x|\lambda) & D_{3,\xi}^{(1)}(x|\lambda) & D_{3,\xi}^{(2)}(x|\lambda) & D_{3,\xi}^{(3)}(x|\lambda)
\end{array}
\right).
\end{eqnarray*}
}

\noindent
where
\begin{eqnarray*}
D_{3,\xi}^{(0)}(x|\lambda) &= &\xi^3 x(x-1)(x-2)\; D_{3,\xi}^{(1)}(x|\lambda)= -\frac{\xi^3}{2} (\lambda-2x+2)(\lambda-2x+x^2-\lambda x),
\\
D_{3,\xi}^{(2)}(x|\lambda) &=& -\frac{\xi^3}{2}(\lambda-x+1)(\lambda^2-4\lambda x+4\lambda+2x^2-4x),
 \\
D_{3,\xi}^{(0)}(x|\lambda) &=& -\frac{\xi^3}{4}(3\lambda-2x+2)(3\lambda^2-6\lambda x+6\lambda+2x^2-4x).
\end{eqnarray*}

\begin{remark}
In fact, it seems that the statement in (\ref{54}) is not correct, the second equation of, Kim et al. \cite[Theorem 1]{Kimetal2013b}.
From (\ref{56}), multiplying both sides by ${\bf \Xi}^{-1}$, we have,
\begin{equation*}
{\bf \Xi}^{-1} \, {\bf D}_\xi^{(k)} (x|\lambda)={\bf \Xi}^{-1} \,  {\bf \Xi S}_1 \,  {\bf \Lambda B}^{(k)} \left(\frac{x}{\lambda} \right)=
{\bf S}_1 \, {\bf \Lambda B}^{(k)} \left( \frac{x}{\lambda} \right),
\end{equation*}

\noindent
then multiplying both sides by ${\bf S}_2$, we have
\begin{equation} \label{58}
{\bf S}_2 \,  {\bf \Xi}^{-1} \, {\bf D}_\xi^{(k)} (x|\lambda)= {\bf S}_2 \,  {\bf S}_1 \, {\bf \Lambda B}^{(k)} \left(\frac{x}{\lambda} \right)={\bf \Lambda B}^{(k)} \left(\frac{x}{\lambda}\right).
\end{equation}

\noindent
From (\ref{54}) and (\ref{58}), it is clear that there is a contradiction.
\end{remark}

\noindent
In the following theorem we obtained the corrected relation as follows.
\begin{theorem}
For $m\in \mathbb{Z},\; k\in \mathbb{N}$, we have
\begin{equation} \label{59}
\lambda^m  B_m^{(k)}  \left(\frac{x}{\lambda} \right)= \sum_{n=0}^m D_{n,\xi}^{(k)} (x| \lambda) \xi^{-n} s_2 (m,n).
\end{equation}
\end{theorem}
\begin{proof}
From Eq. (\ref{50}), replacing  $t$ by $(e^t-1)/ \xi $ ,  we have
\begin{eqnarray} \label{60}
\left( \frac{ \lambda \log⁡\left( 1 + \frac{\xi(e^t-1)}{\xi}\right)}{\left(1+ \frac{\xi(e^t-1)}{\xi} \right)^\lambda-1}\right)^k
\left( 1 + \frac{\xi(e^t-1)}{\xi} \right)^x & = & \sum_{n=0}^\infty D_{n,\xi}^{(k)} (x| \lambda) \frac{ (e^t-1)^n}{n! \xi^n}
\nonumber\\
\left( \frac{\lambda t}{e^{\lambda t}-1} \right)^k e^{tx} & = & \sum_{n=0}^\infty D_{n,\xi}^{(k)} (x|\lambda) \frac{(e^t-1)^n}{n!\xi^n}.
\end{eqnarray}

\noindent
Substituting from (\ref{6}) into (\ref{60}), we have
\begin{eqnarray} \label{61}
\left( \frac{\lambda t}{e^{\lambda t}-1} \right)^k e^{\lambda t(\frac{x}{\lambda})} & = & \sum_{n=0}^\infty D_{n,\xi }^{(k)} (x|\lambda) \xi^{-n} \sum_{m=n}^\infty s_2 (m,n)\frac{t^m}{m!}
\nonumber\\
& = &\sum_{m=0}^ \infty \sum_{n=0}^m D_{n,\xi}^{(k) } (x|\lambda) \xi^{-n}  s_2 (m,n)\frac{t^m}{m!}.
\end{eqnarray}

\noindent
From (\ref{1}) and (\ref{61}), we have
\begin{equation} \label{62}
\sum_{m=0}^\infty \lambda^m B_m^{(k)} \left(\frac{x}{\lambda} \right) \frac{ t^m}{m!} = \sum_{m=0}^\infty \sum_{n=0}^m D_{n,\xi }^{(k)} (x|\lambda) \xi^{-n}  s_2 (m,n) \frac{t^m}{m!}.
\end{equation}

\noindent
Equating the coefficients of $t^m$ on both sides gives (\ref{59}). This completes the proof.
\end{proof}

\noindent
Moreover, we can represent Equation (\ref{59}), in the following matrix form as (\ref{58}).
\begin{equation} \label{63}
{\bf B}^{(k)} \left(\frac{x}{\lambda} \right)= {\bf \Lambda}^{-1} \, {\bf S}_2 \, {\bf \Xi }^{-1} \, {\bf D}_\xi^{(k)} (x|\lambda).                                                                      \end{equation}

\noindent
For example, if setting  $0 \leq n \leq 3,\; 0 \leq k \leq n$, in (\ref{63}), we have
{\footnotesize
\begin{eqnarray*}
\left(
\begin{array}{cccc}
1 & 0 & 0 & 0 \\
0 & \frac{1}{\lambda} & 0 & 0 \\
0 & 0 & \frac{1}{\lambda^2} & 0 \\
0 & 0 & 0 & \frac{1}{\lambda^3}
\end{array}
\right)
\left(
\begin{array}{cccc}
1 & 0 & 0 & 0 \\
0 & 1 & 0 & 0 \\
0 & 1 & 1 & 0 \\
0 & 1 & 3 & 1
\end{array}
\right)
\left(
\begin{array}{cccc}
1 & 0 & 0 & 0 \\
0 & \frac{1}{\xi} & 0 & 0 \\
0 & 0 & \frac{1}{\xi^2} & 0 \\
0 & 0 & 0 & \frac{1}{\xi^3}
\end{array}
\right)\times
\; \; \; \; \; \; \; \; \; \; \; \; \; \; \; \; \; \; \; \; \; \; \; \; \; \; \; \; \; \; \; \;
\; \; \; \; \; \; \; \; \; \; \; \; \; \; \; \; \; \; \; \; \; \; \; \; \; \; \; \; \; \; \; \;
\; \; \; \;
\\
\left(
\begin{array}{cccc}
1 & 1 & 1 & 1 \\
\xi x & -\frac{\xi}{2} (\lambda-2x)&-\xi(\lambda-x)&-\frac{\xi}{2}(3\xi-2x)\\
\xi^2 x(x-1) & \xi^2 (\frac{\lambda^2}{6}-\lambda x+ \frac{\lambda}{2}+x^2-x) &\xi^2 (\frac{5}{6} \lambda^2-2\lambda x+\lambda+x^2-x) &\xi^2 (2\lambda^2-3\lambda x+\frac{3}{2} \lambda+x^2-x)\\
D_{3,\xi}^{(0)}(x|\lambda) & D_{3,\xi}^{(1)}(x|\lambda) & D_{3,\xi}^{(2)}(x|\lambda) & D_{3,\xi}^{(3)}(x|\lambda)
\end{array}
\right)
\\
= \left(
\begin{array}{cccc}
1 & 1 & 1 & 1 \\
\frac{x}{\lambda} & \frac{x}{\lambda}-\frac{1}{2} & \frac{x}{\lambda}-1 & \frac{x}{\lambda}-\frac{3}{2} \\
\frac{x^2}{\lambda^2} & \frac{x^2}{\lambda^2} -\frac{x}{\lambda}+ \frac{1}{6} & \frac{x^2}{\lambda^2} -\frac{2x}{\lambda}+\frac{5}{6} & \frac{x^2}{\lambda^2} -\frac{3x}{\lambda}+2 \\
\frac{x^3}{\lambda^3} & \frac{x}{2\lambda}-\frac{3x^2}{2\lambda^2}+ \frac{x^3}{\lambda^3} & \frac{5x}{2\lambda}-\frac{3x^2}{\lambda^2} + \frac{x^3}{\lambda^3} - \frac{1}{2} & \frac{6x}{\lambda}-\frac{9x^2}{2\lambda^2}+ \frac{x^3}{\lambda^3} -\frac{9}{4}
\end{array}
\right)
\; \; \; \; \; \; \; \; \; \; \; \; \; \; \; \; \; \; \; \; \; \; \; \; \; \; \; \; \; \; \; \;
\end{eqnarray*}
}

\noindent
where
\begin{eqnarray*}
D_{3,\xi}^{(0)}(x|\lambda)& =& \xi^3 x(x-1)(x-2), \; D_{3,\xi}^{(1)}(x|\lambda)=-\frac{\xi^3}{2} (\lambda-2x+2)(\lambda-2x+x^2-\lambda x)
\\
D_{3,\xi}^{(2)}(x|\lambda)&=&-\frac{\xi^3}{2}(\lambda-x+1)(\lambda^2-4\lambda x+4\lambda+2x^2-4x),
\\
D_{3,\xi}^{(3)}(x|\lambda)&=&-\frac{\xi^3}{4}(3\lambda-2x+2)(3\lambda^2-6\lambda x+6\lambda+2x^2-4x).
\end{eqnarray*}

\noindent
Kim et al. \cite{Kimetal2013b} introduced the twisted $\lambda$-Daehee polynomials of the second kind of order $k$ as follows:
\begin{equation} \label{64}
\left( \frac{\lambda  \log⁡(1+\xi t) (1+ \xi t)^\lambda}{(1+\xi t)^\lambda-1} \right)^k (1+ \xi t)^x = \sum_{n=0}^\infty \hat{D}_{n,\xi }^{(k) } (x|\lambda) \frac{ t^n}{n!}.
\end{equation}

\noindent
Setting $x=0, \; \hat{D}_{n,\xi,\lambda}^{(k) }= \hat{D}_{n,\xi }^{(k)} (0|\lambda)$, we have the twisted Daehee numbers of second kind of order $k$.
\begin{equation} \label{65}
\left( \frac{\lambda  \log⁡(1+\xi t) (1+ \xi t)^\lambda}{(1+\xi t)^\lambda-1} \right)^k = \sum_{n=0}^\infty \hat{D}_{n,\xi , \lambda}^{(k) }  \frac{ t^n}{n!}.
\end{equation}

\noindent
Kim et al. \cite[Theorem 2]{Kimetal2013b}, proved that. For $m\in \mathbb{Z}, \; k\in \mathbb{N}$, we have
\begin{equation} \label{66}
\xi^{-m} \hat{D}_{n,\xi} (x│ \lambda )= \sum_{l=0}^m s_1 (m,l)  \lambda^l B_l^{(k)} \left(k+\frac{x}{\lambda} \right),
\end{equation}
and
\begin{equation} \label{67}
\lambda^m B_{m,\xi^\lambda}^{(k)} \left(k+\frac{x}{\lambda}\right)=\sum_{n=0}^m \hat{D}_{n,\xi}^{(k) } (x│ \lambda) s_2 (m,n) \xi^{-n-\lambda k-x}.
\end{equation}

\noindent
Using Eq. (\ref{47}), we can write (\ref{66}) in the following matrix form.
\begin{equation} \label{68}
\hat{\bf D}_\xi^{(k) } (x|\lambda)= {\bf \Xi \,  S}_1 \, {\bf \Lambda }_1 \,  {\bf B}^{(k)} \left(-\frac{x}{\lambda} \right),
\end{equation}

\noindent
where $\hat{\bf D}_\xi^{(k)} (x| \lambda)$ is  $(n+1)\times(k+1)$ matrix for the twisted Daehee numbers of the second kind of the order $k$.

\noindent
For example, if setting $0\leq n \leq 3,\; 0\le k \leq n$, in (\ref{68}), we have
{\footnotesize
\begin{eqnarray*}
\left(
\begin{array}{cccc}
1 & 0 & 0 & 0 \\
0 & \xi & 0 & 0 \\
0 & 0 & \xi^2 & 0 \\
0 & 0 & 0 & \xi^3
\end{array}
\right)
\left(
\begin{array}{cccc}
1 & 0 & 0 & 0 \\
0 & 1 & 0 & 0 \\
0 & -1 & 1 & 0 \\
0 & 2 & -3 & 1
\end{array}
\right)
\left(
\begin{array}{cccc}
1 & 0 & 0 & 0 \\
0 & -\lambda & 0 & 0 \\
0 & 0 & \lambda^2 & 0 \\
0 & 0 & 0 & -\lambda^3
\end{array}
\right)
\times
\; \; \; \; \; \; \; \; \;\; \; \; \; \; \; \; \; \; \; \; \; \; \; \; \; \; \;\; \; \; \; \;
\; \; \; \; \; \; \; \; \;\; \; \; \; \; \; \; \; \; \; \; \; \; \; \; \; \; \;\;
\\
\left(
\begin{array}{cccc}
1 & 1 & 1 & 1 \\
-\frac{x}{\lambda} & - \frac{x}{\lambda}-\frac{1}{2} & -\frac{x}{\lambda}-1 & -\frac{x}{\lambda}- \frac{3}{2} \\
\frac{x^2}{\lambda^2} & \frac{x^2}{\lambda^2} + \frac{x}{\lambda}+\frac{1}{6} & \frac{x^2}{\lambda^2} + \frac{2x}{\lambda}+ \frac{5}{6} & \frac{x^2}{\lambda^2} + \frac{3x}{\lambda}+2 \\
-\frac{x^3}{\lambda^3} & -\frac{x}{2\lambda}-\frac{3x^2}{2\lambda^2}- \frac{x^3}{\lambda^3} & -\frac{5x}{2\lambda}-\frac{3x^2}{\lambda^2} -\frac{x^3}{\lambda^3} -\frac{1}{2} &-\frac{6x}{\lambda}-\frac{9x^2}{2\lambda^2}-\frac{x^3}{\lambda^3} -\frac{9}{4}
\end{array}
\right)
=
\; \; \; \; \; \; \; \; \;\; \; \; \; \; \; \; \; \; \; \; \; \; \; \; \; \; \;\; \; \; \; \; \; \; \;
\; \; \; \; \; \; \; \; \;\; \; \;
\\
\left(
\begin{array}{cccc}
1 & 1 & 1 & 1 \\
\xi x & \frac{\xi}{2} (\lambda+2x)& \xi(\lambda+x) & \frac{\xi}{2}(3\xi+2x)\\
\xi^2 x(x-1)& \xi^2 (\frac{\lambda^2}{6} + \lambda x-\frac{\lambda}{2} + x^2-x)& \xi^2 (\frac{5}{6} \lambda^2+2\lambda x-\lambda+x^2-x)&\xi^2 (2\lambda^2+3\lambda x-\frac{3}{2} \lambda+x^2-x)\\
\hat{D}_{3,\xi}^{(0)}(x| \lambda) & \hat{D}_{3,\xi}^{(1)}(x| \lambda) & \hat{D}_{3,\xi}^{(2)}(x| \lambda) & \hat{D}_{3,\xi}^{(3)}(x| \lambda)
\end{array}
\right).
\end{eqnarray*}
}
where
\begin{eqnarray*}
\hat{D}_{3,\xi}^{(0)}(x| \lambda) &=  & \xi^3 x(x-1)(x-2), \; \hat{D}_{3,\xi}^{(1)}(x| \lambda)= -\frac{\xi^3}{2} (\lambda+2x-2)(\lambda+2x-x^2-\lambda x),
\\
\hat{D}_{3,\xi}^{(2)}(x| \lambda)& = & \frac{\xi^3}{2}(\lambda+x-1)(\lambda^2+4\lambda x-4\lambda +2x^2-4x),
\\
\hat{D}_{3,\xi}^{(3)}(x| \lambda) & = &\frac{\xi^3}{4}(3\lambda+2x-2)(3\lambda^2+6\lambda x-6\lambda+2x^2-4x).
\end{eqnarray*}

\begin{remark}
In fact, it seems that there is something not correct in (\ref{67}), the second equation of Kim et al. \cite[Theorem 2]{Kimetal2013b}.\\
From (\ref{68}), multiplying both sides by ${\bf \Xi}^{-1}$, we have,
\begin{equation*}
{\bf \Xi}^{-1} \,  \hat{\bf D}_\xi^{(k)} (x|\lambda)= {\bf \Xi}^{-1} \,  {\bf \Xi \,  S}_1 {\bf \Lambda}_1 \,  { \bf B}^{(k)} \left(-\frac{x}{\lambda} \right)={\bf S}_1 \,  { \bf \Lambda}_1 \,  {\bf B}^{(k) } \left(-\frac{x}{\lambda} \right)
\end{equation*}

\noindent
multiplying both sides by ${\bf S}_2$, we have
\begin{equation} \label{70}
{\bf S}_2 \,  {\bf \Xi}^{-1} \,  \hat{\bf D}_\xi^{(k)} (x|\lambda)= {\bf S}_2 \, {\bf S}_1 \,  { \bf \Lambda }_1 \, {\bf B}^{(k) } \left(-\frac{x}{\lambda} \right)= {\bf I} \, { \bf \Lambda }_1 \, {\bf B}^{(k) } \left(-\frac{x}{\lambda} \right)={ \bf \Lambda}_1 \, {\bf B}^{(k) } \left(-\frac{x}{\lambda} \right).
\end{equation}

\noindent
From (\ref{67}) and (\ref{70}) there is a contradiction.
\end{remark}

\noindent
We obtained the corrected relation in the following theorem as follows.
\begin{theorem}
For $m\in \mathbb{Z}, \; k\in \mathbb{N}$, we have
\begin{equation} \label{71}
\lambda^m B_m^{(k)} \left(k+ \frac{x}{\lambda} \right)= \sum_{n=0}^m \hat{D}_{n,\xi}^{(k)} (x|\lambda) \xi^{-n}  s_2 (m,n).
\end{equation}
\end{theorem}
\begin{proof}
From Eq. (\ref{64}), replacing  $t$ by $(e^t-1)/\xi$ ,  we have
\begin{equation*}
\left(\frac{\lambda  \log⁡\left(1+ \frac{\xi(e^t-1)}{\xi}\right) \left(1+\frac{\xi(e^t-1)}{\xi}\right)^\lambda}{\left(1+\frac{\xi(e^t-1)}{\xi} \right)^\lambda-1} \right)^k \left(1+\frac{\xi(e^t-1)}{\xi} \right)^x =
\sum_{n=0}^\infty \hat{D}_{n,\xi}^{(k)} (x|\lambda) \frac{ (e^t-1)^n}{n! \xi^n},
\end{equation*}
\begin{eqnarray} \label{72}
\left( \frac{\lambda t}{e^{\lambda t}-1} \right)^k e^{(k \lambda+x)t} & = & \sum_{n=0}^\infty \hat{D}_{n,\xi}^{(k)} (x| \lambda)  \frac{(e^t-1)^n}{n! \xi^n},
\nonumber\\
\left(\frac{\lambda t}{e^{\lambda t}-1} \right)^k e^{\lambda t\left(k+ \frac{x}{\lambda} \right)} & = & \sum_{n=0}^\infty \hat{D}_{n,\xi}^{(k)} (x| \lambda)  \xi^{-n} \frac{(e^t-1)^n}{n!}.
\end{eqnarray}

\noindent
Substituting from Eq. (\ref{6}) into (\ref{72}), we have
\begin{eqnarray} \label{73}
\left(\frac{\lambda t}{(e^{\lambda t}-1} \right)^k e^{\lambda t\left(k+\frac{x}{\lambda} \right)} & = &
\sum_{n=0}^\infty \hat{D}_{n,\xi}^{(k)} (x|\lambda) \xi^{-n} \sum_{m=n}^\infty s_2 (m,n) \frac{ t^m}{m!}
\nonumber\\
&= & \sum_{m=0}^\infty \sum_{n=0}^m \hat{D}_{n,\xi}^{(k)} (x|\lambda) \xi^{-n}  s_2 (m,n) \frac{t^m}{m!}.
\end{eqnarray}

\noindent
From Eq. (\ref{1}) and (\ref{73}), we have
\begin{equation} \label{74}
\sum_{m=0}^\infty \lambda^m B_m^{(k)} \left(k+\frac{x}{\lambda} \right) \frac{ t^m}{m!}=
\sum_{m=0}^\infty \sum_{n=0}^m \hat{D}_{n,\xi}^{(k)} (x|\lambda) \xi^{-n}  s_2 (m,n) \frac{t^m}{m!}.
\end{equation}
Equating the coefficients of $t^m$ on both sides gives (\ref{71}). This completes the proof.
\end{proof}

\noindent
Moreover, by using Eq. (\ref{47}), we can represent Equation (\ref{71}), in the following matrix form.
\begin{equation} \label{75}
{\bf B}^{(k)} \left(-\frac{x}{\lambda}\right)= {\bf \Lambda}_1^{-1} \,  {\bf S}_2 \,  {\bf \Xi }^{-1} \,  \hat{ \bf D}_\xi^{(k)} (x|\lambda),
\end{equation}
where ${\bf \Lambda \,  B}^{(k) } \left(k+\frac{x}{\lambda}\right)={\bf \Lambda}_1 \,  {\bf B}^{(k)} \left(-\frac{x}{\lambda} \right)$.

\noindent
For example, if setting  $0\leq n \leq 3,\, 0 \leq k \leq n$, in (\ref{75}), we have
{\footnotesize
\begin{eqnarray*}
\left(
\begin{array}{cccc}
1 & 0 & 0 & 0 \\
0 & -\frac{1}{\lambda} & 0 & 0 \\
0 & 0 & \frac{1}{\lambda^2} & 0 \\
0 & 0 & 0 & -\frac{1}{\lambda^3}
\end{array}
\right)
\left(
\begin{array}{cccc}
1 & 0 & 0 & 0 \\
0 & 1 & 0 & 0 \\
0 & 1 & 1 & 0 \\
0 & 1 & 3 & 1
\end{array}
\right)
\left(
\begin{array}{cccc}
1 & 0 & 0 & 0 \\
0 & \frac{1}{\xi} & 0 & 0 \\
0 & 0 & \frac{1}{\xi^2} & 0 \\
0 & 0 & 0 & \frac{1}{\xi^3}
\end{array}
\right)\times
\; \; \; \; \; \; \; \; \; \; \; \; \; \; \; \; \; \; \; \; \; \; \; \; \; \; \; \; \; \; \;
\; \; \; \; \; \; \; \; \; \; \; \; \; \; \; \; \; \; \; \; \; \; \; \; \; \; \; \; \; \; \;
\\
\left(
\begin{array}{cccc}
1 & 1 & 1 & 1 \\
\xi x & \frac{\xi}{2} (\lambda+2x) & \xi(\lambda+x)& \frac{\xi}{2}(3\xi+2x)\\
\xi^2 x(x-1)& \xi^2 (\frac{\lambda^2}{6} + \lambda x- \frac{\lambda}{2} + x^2-x)& \xi^2 (\frac{5}{6} \lambda^2+2\lambda x-\lambda+x^2-x)& \xi^2 (2\lambda^2+3\lambda x-\frac{3}{2} \lambda + x^2-x)\\
\hat{D}_{3,\xi}^{(0)}(x|\lambda) & \hat{D}_{3,\xi}^{(1)}(x|\lambda) & \hat{D}_{3,\xi}^{(2)}(x|\lambda) & \hat{D}_{3,\xi}^{(3)}(x|\lambda)
\end{array}
\right)
\\
=
\left(
\begin{array}{cccc}
1 & 1 & 1 & 1 \\
- \frac{x}{\lambda} &- \frac{x}{\lambda}- \frac{1}{2} & - \frac{x}{\lambda}-1 &- \frac{x}{\lambda}-\frac{3}{2} \\
\frac{x^2}{\lambda^2} & \frac{x^2}{\lambda^2} + \frac{x}{\lambda}+ \frac{1}{6} & \frac{x^2}{\lambda^2} + \frac{2x}{\lambda}+ \frac{5}{6} & \frac{x^2}{\lambda^2} + \frac{3x}{\lambda}+2\\
-\frac{x^3}{\lambda^3} & -\frac{x}{2\lambda}-\frac{3x^2}{2\lambda^2}-\frac{x^3}{\lambda^3} & -\frac{5x}{2\lambda}-\frac{3x^2}{\lambda^2} -\frac{x^3}{\lambda^3} -\frac{1}{2} &- \frac{6x}{\lambda}-\frac{9x^2}{2\lambda^2}-\frac{x^3}{\lambda^3} -\frac{9}{4}.
\end{array}
\right)
\; \; \; \; \; \; \; \; \; \; \; \; \; \; \; \; \; \; \; \; \; \; \; \; \; \; \; \; \; \; \;
\end{eqnarray*}
}
where
\begin{eqnarray*}
\hat{D}_{3,\xi}^{(0)}(x|\lambda)
&= &\xi^3 x(x-1)(x-2), \; \hat{D}_{3,\xi}^{(1)}(x|\lambda)= -\frac{\xi^3}{2} (\lambda+2x-2)(\lambda+2x-x^2-\lambda x),
\\
\hat{D}_{3,\xi}^{(2)}(x|\lambda)&= &\frac{\xi^3}{2}(\lambda+x-1)(\lambda^2+4\lambda x-4\lambda +2x^2-4x),
\\
\hat{D}_{3,\xi}^{(3)}(x|\lambda)&= & \frac{\xi^3}{4}(3\lambda+2x-2)(3\lambda^2+6\lambda x-6\lambda+2x^2-4x).
\end{eqnarray*}


\bibliographystyle{plain}

\end{document}